\documentclass{siamltex}
\usepackage[numbers]{natbib}
\usepackage{amsmath}
\usepackage{amscd}
\usepackage{amssymb}
\usepackage{graphicx}
\usepackage{url}
\usepackage{subcaption}
\usepackage{color}
\usepackage{listings}
\usepackage{booktabs}
\usepackage{multirow}
\usepackage{textcomp}
\usepackage{mathtools}
\usepackage{mcode}
\usepackage{hyperref}
\usepackage[mathscr]{euscript}
\usepackage{relsize}

\definecolor{DarkBlue}{rgb}{0.00,0.00,0.55}
\definecolor{Black}{rgb}{0.00,0.00,0.00}
\hypersetup{
    linkcolor = DarkBlue,
    anchorcolor = DarkBlue,
    citecolor = DarkBlue,
    filecolor = DarkBlue,
    urlcolor = DarkBlue,
    colorlinks  = true,
}

\title{Distinct solutions of finite-dimensional complementarity problems}

\author{
  M.~Croci\thanks{Mathematical Institute, University of Oxford, Oxford, UK
    (\texttt{matteo.croci@maths.ox.ac.uk}).}
 \and
  P.~E.~Farrell\thanks{Mathematical Institute, University of Oxford, Oxford, UK.
    Center for Biomedical Computing, Simula Research Laboratory, Oslo, Norway
    (\texttt{patrick.farrell@maths.ox.ac.uk}).\newline
    This research is funded by EPSRC grants EP/K030930/1, EP/M019721/1,
    and a Center of Excellence grant from the Research Council of Norway to the Center for Biomedical Computing at Simula Research Laboratory. The authors would like to acknowledge useful discussions with T.~S.~Munson and
    N.~I.~M.~Gould.
    }
  }

\begin{document}

\maketitle

\begin{abstract}
Complementarity problems often permit distinct solutions, a fact of major
significance in optimization, game theory and other fields. In this paper, we
develop a numerical technique for computing multiple isolated solutions of
complementarity problems, starting from the same initial guess. This technique,
called deflation, is applied in conjunction with existing algorithms that reformulate
the complementarity problem as the rootfinding problem of a semismooth residual.
After one solution is found, the idea of deflation is to apply operators to the
arguments of the corresponding semismooth reformulation to ensure that
solvers will not converge to that same solution again. This ensures that if the semismooth
solver is restarted from the same initial guess and it converges, it will
converge to a different solution.
We prove theoretical results on the effectiveness of the method, and
apply it to several difficult finite-dimensional complementarity problems from the
literature. While deflation is not guaranteed to find all solutions, for every problem considered with a finite number of solutions, we identify
initial guesses from which all known solutions are computed with deflation.
\end{abstract}



\section{Introduction}
Complementarity problems are an important generalisation of systems of nonlinear
equations that incorporate inequality constraints. They arise in many areas of applied
mathematics, most prominently as the Karush--Kuhn--Tucker (KKT) first order optimality
conditions for optimization with inequality constraints. They also have
important applications in contact mechanics, game theory, economics, finance,
fracture mechanics, and obstacle problems \cite{ferris1997}.

Complementarity problems often admit multiple solutions, which are
typically significant for the application at hand. For example, a nonconvex
optimization problem can permit several local minima, while a bimatrix game can
permit multiple Nash equilibria. In this paper, we develop a novel numerical
technique that can successfully identify multiple solutions of complementarity
problems, provided they exist and are isolated from each other. Our approach builds upon
existing state-of-the-art complementarity solvers, to enable them to find
multiple solutions of complementarity problems starting from the same initial
guess.

The technique we develop is called deflation. The first algorithm in this
spirit was designed by Wilkinson \cite{wilkinson1963} to find distinct roots of
polynomials. The basic idea of deflation is: given a problem and one of its
solutions, construct a new problem which retains all solutions except for the
one deflated. This ensures that different solutions can be
identified if the algorithm is restarted after each deflation. For example,
given a polynomial $p(x)$ and a root $r$, one may form the
deflated polynomial $g(x) = p(x)/(x-r)$ and apply Newton's method to $g$.
Deflation was extended to nonlinear algebraic systems by Brown and Gearhart
\cite{brown1971} and to nonlinear partial differential equations by Farrell et
al.\ \cite{farrell2014}. The aim of this paper is to extend this approach to
finite-dimensional complementarity problems.

In this work we show that the standard deflation techniques introduced by Brown and
Gearhart and Farrell et al.\ are not sufficient to ensure nonconvergence to
known solutions for complementarity problems. We thus define a new class of
deflation operators, called complementarity deflation operators, and construct
instances of these operators with numerically desirable properties. We then
prove 
theoretical results about the effectiveness of these operators at eliminating
known solutions.

The importance of multiple solutions of complementarity problems has motivated
other authors to develop various approaches for computing them. A simple
strategy is to vary the initial guess given to the solver
\cite{vonstengel2002}, but this is heuristic and labour-intensive
\cite{tinloi2003}. Judice and Mitra \cite{judice1988} develop an algorithm for
enumerating the solutions of linear complementarity problems. Their
algorithm requires exhaustive exploration of a binary tree whose size is
exponential in terms of the size of the problem, and is thus impractical for
large problems. Tin-Loi and Tseng \cite{tinloi2003} develop an algorithm for
finding multiple solutions of linear complementarity problems by augmenting the
problem with constraints that eliminate known solutions; while very successful
on the problems considered, the size of each problem increases with each
solution eliminated. By contrast, the technique presented here does not increase
the size of the problems to be solved after each solution found. The closest
previous work is that of Kanzow \cite{kanzow2000}, who uses a similar idea to
improve the convergence of a semismooth solver to a single solution.  If the
semismooth solver runs into difficulty at a point, then Kanzow applies a
standard deflation operator to the semismooth
residual to encourage the solver to escape from the difficult point. As we
demonstrate later, the operator applied is not sufficient to guarantee
nonconvergence to points deflated with it. However, the general idea of applying
deflation to avoid points at which the solver performs poorly is a useful one,
and will be exploited in a later example.

Merely removing known solutions from consideration might not be sufficient to
make deflation a practical technique for computing distinct solutions of
difficult complementarity problems: after all, deflation guarantees
nonconvergence to known solutions, but does not guarantee convergence to unknown
solutions. We therefore investigate the effectiveness of the technique on difficult
algebraic complementarity problems from the literature, in combination with a
standard semismooth complementarity solver \cite{facchinei1997b,sun2002}.  For
all considered problems with a finite solution set, we identify at least
one initial guess that converges to all known solutions, demonstrating its
potential.

The source code for the semismooth solver, deflation algorithm and the
computational examples is included as supplementary material.

\section{Background on complementarity and deflation}
\subsection{Complementarity}
A complementarity problem is a generalisation of a nonlinear system of equations
which is defined by a vector function $F(z)$, which we call the problem residual,
and by lower and upper bounds for the variable $z$.
\begin{definition}[Mixed complementarity problem (MCP)]
Let $F:\mathbb{R}^n\rightarrow \mathbb{R}^n$ be the
problem residual. Given the
lower and upper bounds, 
\begin{align*}
l \in (\mathbb{R}\cup\{-\infty\})^n,\hspace{12pt}u \in(\mathbb{R}\cup\{+\infty\})^n,
\end{align*}
with $l_i \leq u_i$ for all $i$, then the mixed complementarity problem
MCP$(F,l,u)$ is to find $z\in \mathbb{R}^n$ such that for each $i$, one
of the following conditions holds:
\begin{align}
l_i = z_i\phantom{u_i =\hspace{1mm} } &\text{ and } F_i(z) \ge 0, \nonumber \\
l_i < z_i < u_i           &\text{ and } F_i(z) = 0,             \\
\phantom{l_i < }z_i = u_i &\text{ and } F_i(z) \le 0. \nonumber
\end{align}
The set $\mathscr{F}=\{z\in \mathbb{R}^n: l_i\leq z_i \leq
u_i \text{ for all } i\}$ is called the feasible set of MCP$(F,l,u)$.
\end{definition}

A simple example of an MCP arising in practice can be found when looking for
local minima of a continuously differentiable function
$f:\mathbb{R}^n\rightarrow \mathbb{R}$ in the feasible region
$\mathscr{F}=\{z\in \mathbb{R}^n: l_i\leq z_i \leq u_i\hspace{6pt}$for all
$i\}$. The KKT optimality conditions are MCP$(\nabla f,l,u)$. MCPs also arise
as optimality conditions for more complicated optimization problems involving
general inequality constraints \cite{nocedal2006}.

An important specialisation of an MCP is when the lower and upper bounds are
zero and infinity respectively. This problem is then called a nonlinear
complementarity problem.
\begin{definition}[Nonlinear complementarity problem (NCP)]
\label{def:NCP_def}
The nonlinear complementarity problem NCP$(F)$ is equivalent to
MCP$(F,0,\infty)$ and can be expressed as follows. Find $z\in \mathbb{R}^n$ such
that:
\begin{flalign}
\label{eqn:ncp}
0\leq z \perp F(z)\geq 0,
\end{flalign}
where $\perp$ signifies that $z_iF_i(z)=0
$ for all $i$.
\end{definition}
For simplicity, we develop our theory of complementarity deflation
in the context of NCPs, and later extend our approach to the case of general
MCPs.

The class of solvers considered in this work relies on the reformulation of an
NCP as a semismooth rootfinding problem. Central to this reformulation is the
concept of an NCP function.
\begin{definition}[NCP function]
A function $\phi:\mathbb{R}^2\rightarrow \mathbb{R}$ is an NCP function if for all $(a,b)\in \mathbb{R}^2$ it satisfies
\begin{align*}
\phi(a,b)=0\hspace{6pt}\Longleftrightarrow\hspace{6pt}0\leq a \perp b \geq 0.
\end{align*}
\end{definition}
The NCP function used in our computations is the Fischer-Burmeister function
\cite{fischer1992},
\begin{align}
\phi_{FB}(a,b)=\sqrt{a^2+b^2}-a-b.
\end{align}
An NCP function can be used to construct an operator for the reformulation of
complementarity problems.
\begin{definition}[NCP operator]
An operator $\Phi:\mathbb{R}^n\times \mathbb{R}^n\rightarrow \mathbb{R}^n$ is an
NCP operator if for all $z,w\in \mathbb{R}^n$ it satisfies
\begin{align}
\label{eq:NCP_operator_definition}
\Phi(z,w)=0\hspace{6pt}\Longleftrightarrow\hspace{6pt}0\leq z \perp w \geq 0.
\end{align}
\end{definition}
For example, an NCP operator can be defined via any NCP function by
\begin{align}
\label{eq:NCP_operator_particular}
\Phi_i(z,w)=\phi(z_i,w_i).
\end{align}
A candidate $z$ is a solution of NCP$(F)$ if and only if it is a root
of $\Phi(z, F(z))$. Thus, for a given $F$, we can define the associated
NCP residual.
\begin{definition}[NCP residual]
Given an NCP$(F)$, its NCP residual is
\begin{align}
\Psi(z) = \Phi(z, F(z)).
\end{align}
\end{definition}
In this manner, solving the NCP can be reformulated as finding the roots of
an associated NCP residual.

In this work we set $\Phi$ and $\Psi$ to be the NCP operator and residual induced by the
Fischer--Burmeister NCP function. $\Psi$ is not continuously differentiable, but
is semismooth, and a semismooth generalisation of Newton's method can be
employed to find its roots \cite{ulbrich2011}. The solver we use in our later
examples is the semismooth Newton solver of Facchinei et al.\
\cite{facchinei1997b} augmented with the projected linesearch method of Sun et
al.\ \cite{sun2002} to obtain a feasible method.

\subsection{Deflation of nonlinear equations}
Another important specialisation of the mixed complementarity problem is the
choice $l = -\infty, u = \infty$. This yields the problem of solving a
nonlinear equation, or nonlinear rootfinding: find $z \in \mathbb{R}^n$ such
that $F(z) = 0$.  We now briefly present the theory of deflation for nonlinear
rootfinding, as described in Brown and Gearhart and Farrell et al.\
\cite{brown1971,farrell2014}.

Given a residual $F: \mathbb{R}^n \rightarrow \mathbb{R}^n$ and a solution $r
\in \mathbb{R}^n$ such that $F(r) = 0$ and its Jacobian $F'(r)$ is nonsingular, the deflation
technique constructs a modified residual $G$ with two properties: first, a
candidate solution $z \neq r$ is a root of $F$ if and only if it is
a root of $G$; and second, that standard nonlinear rootfinding methods such
as Newton's method will not converge to $r$ again. This $G$ is constructed
with a deflation operator, which we now define.
\begin{definition}[Deflation operator \cite{brown1971}]
For $r \in \mathbb{R}^n, z \in \mathbb{R}^n \setminus \{r\}$, let
$M(z; r) \in \mathbb{R}^{n \times n}$. We say $M$ is a deflation operator if
$M(z; r)$ is invertible for all $z$ and $r$, and for all $r$ such that $F(r) = 0$
and $F'(r)$ is nonsingular, we have
\begin{align}
\liminf\limits_{z_k \rightarrow r} \|M(z_k; r) F(z_k)\| > 0,
\end{align}
for any sequence $\{z_k\}$ converging to $r$.
\end{definition}
The deflated residual is constructed via $G(z) = M(z; r) F(z)$. Essentially, a
deflation operator $M$ eliminates a root $r$ from consideration by ensuring that
the norm of the deflated residual does not converge to zero along any sequence
converging to $r$; since rootfinding algorithms seek sequences converging to zero, they will not yield a sequence converging to $r$ again. We will 
refer to these operators as standard deflation operators, to contrast them with
the stronger complementarity deflation operators developed in this work.

A typical example of a deflation operator is norm deflation,
\begin{equation} \label{eqn:norm_deflation}
G(z) = M(z; r)F(z) = \frac{I}{\|z - r\|} F(z),
\end{equation}
which is the deflation operator used to aid convergence in Kanzow \cite{kanzow2000}.
Farrell et al.\ propose a shifted norm deflation operator
\begin{equation} \label{eqn:standard_deflation}
G(z) = M(z; r)F(z) = \left(\frac{I}{\|z - r\|^p} + \alpha\right) F(z),
\end{equation}
for a shift $\alpha \geq 0$ and power $p \geq 1$, as this has desirable
numerical properties far from previously found solutions for $\alpha \neq 0$.
As $\|z - r\| \rightarrow \infty$, the deflated residual
behaves asymptotically as $\alpha F(z)$. With $\alpha = 0$, the deflated residual
would approach zero and solvers would erroneously recognize values of $z$ far
from $r$ as solutions.

\subsection{Standard deflation operators fail on complementarity problems}
\label{sec:counterexample}
We close this section with a counterexample demonstrating that standard deflation
operators (developed for nonlinear equations) do not work on nonlinear
complementarity problems. More precisely, even after applying a standard deflation
operator to the NCP residual, there exist sequences $z_k \rightarrow r$ for which the NCP
residual still converges to zero, as the Jacobian of the NCP residual may not
exist or may not be invertible. Following \cite{kanzow2000}, let us apply the
standard norm deflation operator \eqref{eqn:norm_deflation} to $\Psi$:
\begin{equation} \label{eqn:standard_doesnt_work}
\hat{\Psi}(z) = M(z; r)\Phi(z, F(z)) = \frac{\Phi(z, F(z))}{\|z - r\|}.
\end{equation}
The property we desire is: for any problem NCP$(F)$, solution $r$, and
sequence $z_k \rightarrow r$, we hope that
\begin{equation}
\liminf \limits_{z_k \rightarrow r} \|\hat{\Psi}(z_k)\| > 0.
\end{equation}
We now give an example $F, r$ and $\{z_k\}$ for which this does not
hold with standard deflation operators.

First note that with the Fischer--Burmeister NCP function, applying a standard
deflation operator to the NCP residual is equivalent to applying it to its
arguments, as
\begin{equation}
m\phi(a, b) = \sqrt{(ma)^2 + (mb)^2} - ma - mb = \phi(ma, mb),
\end{equation}
where $m \in \mathbb{R}$.
Let $z = [x, y]^T, F(z) = [y + y^2, y + x + 1]^T$
and consider NCP$(F)$. This problem admits $r = [1, 0]^T$ as a solution, with residual $[0,
2]^T$, so that $\Phi(r, F(r)) = 0$.
Thus, after deflation of $r$, we have a deflated position argument
\begin{equation}
H(z) = \left[\dfrac{x}{\sqrt{(x-1)^2+y^2}},\dfrac{y}{\sqrt{(x-1)^2+y^2}}\right]^T,\\
\end{equation}
and a deflated residual argument
\begin{equation}
G(z) = \left[\dfrac{y+y^2}{\sqrt{(x-1)^2+y^2}},\dfrac{y+x+1}{\sqrt{(x-1)^2+y^2}}\right]^T,
\end{equation}
and can write the deflated NCP residual as
\begin{equation}
\hat{\Psi}(z) = \Phi(H(z), G(z)).
\end{equation}
Let us now take the limit as $z_k \rightarrow r$. $H_1$ and $G_2$ diverge to
infinity in the limit as the numerator is bounded away from zero and the denominator
converges to zero.
Since $y_k \rightarrow 0$, $G_1 \sim H_2$ as $z_k \rightarrow
r$. Now observe that, according to the path chosen by the converging sequence,
we have
\begin{align*}
\lim\limits_{z_k\rightarrow r} G_1(z_k)=\lim\limits_{z_k\rightarrow r} H_2(z_k)=
\begin{dcases}\begin{array}{ll}
c\neq 0, & \text{if }\lim\limits_{z_k\rightarrow r}\frac{|y_k|}{(x_k-1)^2} >0,\\
0, & \text{if }\lim\limits_{z_k\rightarrow r}\frac{|y_k|}{(x_k-1)^2}=0.
\end{array}
\end{dcases}
\end{align*}
Hence there exist paths that have in the limit
$G_1H_1 = G_2H_2 = 0$; take for example the path 
$x_k \downarrow 1, y_k = 0$.
Hence, if the solver takes any of these paths, it can converge to the same solution
again: standard deflation operators fail to prevent all the entries of
$\hat{\Psi}$ from converging to zero as $z_k \rightarrow r$.

This example shows that we must augment the requirements of standard deflation
operators.  We now proceed to construct a stronger class of deflation operators,
complementarity deflation operators, that do work for complementarity problems.

\section{Deflation for complementarity problems}
The basic paradigm adopted is the following: we wish to construct operators $H$ and
$G$ such that
\begin{equation}
\liminf \limits_{z_k \rightarrow r} \|\hat{\Psi}(z_k)\| = \liminf \limits_{z_k \rightarrow r} \|\Phi\left(H(z_k; r), G(F, z_k; r)\right)\| > 0,
\end{equation}
and which preserve solutions other than $r$: for $z \neq r$,
\begin{equation}
\hat{\Psi}(z) = 0 \iff \Psi(z) = 0.
\end{equation}
The essential problem of the previous example was that standard deflation
operators only guarantee that \emph{at least one} component of the object
deflated will not converge to zero as $z_k \rightarrow r$.  That is,
\begin{equation} \label{eqn:weak_requirement}
\liminf \limits_{z_k \rightarrow r} \|M(z_k; r)F(z_k)\| > 0 \iff 
\exists j \text{ s.t. } \liminf \limits_{z_k \rightarrow r} |(M(z_k; r)F(z_k))_j| > 0,
\end{equation}
and while at least one component $j$ does not converge to zero, others might,
depending on the particular path taken to approach $r$.
The first component of the deflated position vector $(Mz)_1$ and the second
component of the deflated problem residual $(MF)_2$ did not converge to zero, but their
respective multiplicands $(MF)_1$ and $(Mz)_2$ did converge to zero; their products
converged to zero, and so the NCP residual converged to zero overall. Hence, we
must choose $H$ and $G$ so that \emph{at least one product} $H_jG_j$ does not
converge to zero.

The simplest way to achieve this property is to ensure that at least one
component of the deflated problem residual $G$ does not converge to zero, and that \emph{all} components
of the deflated position vector $H$ do not converge to zero\footnote{One could choose to swap
the requirements for the arguments, but the chosen strategy is more
straightforward. $G$ acts on the problem residual $F$, an arbitrary
(differentiable) function, while $H$ acts only on the identity function over the
feasible region of the NCP, and hence $H$ is easier to constrain.}.

First let us consider the deflation of the problem residual, $G(F, z; r)$.
We summarize the properties we require of $G$ in the following definition.
\begin{definition}[Weak complementarity deflation operator]
Let $\mathscr{F}$ be the feasible region of an NCP, i.e.
\begin{equation}
\label{eq:feasible_region}
\mathscr{F} = \{z\in \mathbb{R}^n: z \ge 0\},
\end{equation}
and let $C^1(\mathscr{F})$ be the space of continuously-differentiable functions
on $\mathscr{F}$.

We say $G: C^1(\mathscr{F}) \times \mathscr{F} \times \mathscr{F} \rightarrow
\mathbb{R}^n$ is a weak complementarity deflation operator if, for any $F \in
C^1(\mathscr{F})$ and for any $r \in \mathscr{F}, z \in \mathscr{F} \setminus \{r\}$ there exists at least one $j \in \{1,
\dots, n\}$ such that
\begin{equation} \label{eqn:weak1}
\liminf \limits_{z_k \rightarrow r} |G_j(F, z_k; r)| > 0,
\end{equation}
and if
\begin{equation}
\label{eq:weak_op_sign_preservation}
\mathrm{sign}[F(z)] = \mathrm{sign}[G(F, z; r)] \text{ for all } z \in \mathscr{F} \setminus \{r\}.
\end{equation}
\end{definition}
Standard deflation operators satisfy \eqref{eqn:weak1} if $F'(r)$ is nonsingular. The standard deflation operators
\eqref{eqn:norm_deflation} and \eqref{eqn:standard_deflation} are also
sign-preserving, and so these operators are also weak complementarity deflation
operators. In practice, it is only necessary to apply a weak complementarity
deflation operator to $F$ if $F(r) = 0$; if $F(r) \neq 0$, then $G(F, z; r) = F(z)$
satisfies the requirements of a weak complementarity deflation operator.

We now turn our attention to the deflated position vector. We summarize the
properties we seek of $H(z; r)$ in the following definition.
\begin{definition}[Complementarity deflation operator]
\label{def:compl_defl_op}
Let $F \in C^1(\mathscr{F})$, where $\mathscr{F}$ is the feasible region of the NCP \eqref{eq:feasible_region}.
Suppose NCP$(F)$ has isolated solutions: that is, the minimum distance
between solutions $\delta_r \in (0, \infty]$. Choose $\delta \in (0, \delta_r]$. Let $B(r, \delta)$ be the open ball of
radius $\delta$ centred at $r$ that excludes $r$, i.e.
\begin{equation}
B(r, \delta) = \{z \in \mathscr{F}: \|z - r\| \in (0, \delta)\}.
\end{equation}

We say $H: \mathscr{F} \times \mathscr{F} \rightarrow \mathbb{R}^n$ is a complementarity deflation operator
if for $r \in \mathscr{F}, z_k \in \mathscr{F} \setminus \{r\}$,
\begin{equation}
\label{eq:compl_defl_bounded_away}
\liminf \limits_{z_k \rightarrow r} |H_j(z_k; r)| > 0 \text{ for all } j,
\end{equation}
and
\begin{align}
\label{eq:comp_defl_sign_preservation}
\mathrm{sign}[z_j] &=   \mathrm{sign}[H_j(z; r)] \text{ for all } z \in \mathscr{F} \setminus B(r, \delta), \nonumber \\
\mathrm{sign}[z_j] &\le \mathrm{sign}[H_j(z; r)] \text{ for all } z \in \mathscr{F} \cap B(r, \delta).
\end{align}
\end{definition}

Our first task is to show that the application of deflation to
an NCP preserves other solutions.
\begin{lemma}[Preservation of other solutions] \label{lem:sol_preservation}
Let $r$ be a solution of NCP$(F)$.
Then, for all $z
\in \mathscr{F} \setminus \{r\}$, we have
\begin{equation}
\Phi(z, F(z)) = 0 \iff \Phi(H(z; r), G(F, z; r)) = 0.
\end{equation}
\end{lemma}
\begin{proof}
As the solutions are isolated, there exists $\delta_r>0$ such that
the region $B(r, \delta_r)$ contains no solutions, and hence $B(r, \delta) \subseteq B(r, \delta_r)$
also contains no solutions.
First, consider $z\in\mathscr{F} \setminus B(r, \delta)$, $z \neq r$.
Properties \eqref{eq:comp_defl_sign_preservation} and \eqref{eq:weak_op_sign_preservation}
ensure that in this region, for all $j$,
\begin{equation*}
\mathrm{sign}[F_j(z)]=\mathrm{sign}\left[G_j(F, z; r)\right], \text{ and } \mathrm{sign}[z_j]=\mathrm{sign}\left[H_j(z; r)\right].
\end{equation*}
This implies that, for all $j$,
\begin{equation*}
0\leq z_j \perp F_j(z)\geq 0 \iff 0\leq H_j(z; r) \perp G_j(F, z; r)\geq 0.
\end{equation*}
Hence, the definition of an NCP function $\phi$ ensures that, for all $j$,
\begin{equation*}
\phi(z_j,F_j(z))=0 \iff \phi(H_j(z; r), G_j(F, z; r))=0,
\end{equation*}
and the result follows for $\mathscr{F} \setminus B(r, \delta)$, $z \neq r$.

Now consider $z \in \mathscr{F} \cap B(r, \delta)$. Since by assumption
$\Phi(z, F(z)) \neq 0$, we must show that $\Phi(H(z; r), G(F, z; r)) \neq 0$ also.
Properties \eqref{eq:comp_defl_sign_preservation} and \eqref{eq:weak_op_sign_preservation}
ensure that in this region, for all $j$,
\begin{equation*}
\mathrm{sign}[F_j(z)]=\mathrm{sign}\left[G_j(F, z; r)\right] \text{ and } \mathrm{sign}[z_j] \le \mathrm{sign}\left[H_j(z; r)\right].
\end{equation*}
As $\Phi(z, F(z)) \neq 0$, there exists at least one $j$ for which
\begin{equation}
\phi(z_j, F_j(z)) \neq 0.
\end{equation}
If $z_j = 0$,    then $F_j(z) < 0$.    Hence $H_j \ge 0$ and $G_j < 0$, so $\phi(H_j, G_j) \neq 0$.
If $z_j \neq 0$, then $F_j(z) \neq 0$. Hence $H_j > 0$   and $G_j \neq 0$, so $\phi(H_j, G_j) \neq 0$
also.
\end{proof}

We now show that applying this deflation strategy does achieve our desired
objective of eliminating deflated solutions from consideration.
\begin{theorem}[Complementarity deflation]
Let $F \in C^1({\mathscr{F}})$ such that NCP$(F)$ has isolated solutions. Let $r$
be a solution of NCP$(F)$. Let $G(F, z; r)$ be a weak
complementarity deflation operator and let $H(z; r)$ be a complementarity
deflation operator. Define the deflated NCP residual as
\begin{equation}
\hat{\Psi}(z) = \Phi(H(z; r), G(F, z; r)).
\end{equation}
Then
\begin{equation}
\liminf \limits_{z_k \rightarrow r} \|\hat{\Psi}(z_k)\| > 0.
\end{equation}
\end{theorem}
\begin{proof}
By the properties of the weak complementarity deflation operator $G$, for any sequence $z_k \rightarrow r$ there exists at least
one index $j$ such that
\begin{equation} \label{eqn:star1}
\liminf \limits_{z_k \rightarrow r} |G_j(F, z_k; r)| > 0.
\end{equation}
By the properties of the complementarity deflation operator $H$,
\begin{equation} \label{eqn:star2}
\liminf \limits_{z_k \rightarrow r} |H_j(z_k; r)| > 0,
\end{equation}
for the same index $j$.
Now we suppose
\begin{equation}
\liminf \limits_{z_k \rightarrow r} \|\hat{\Psi}(z_k)\| = 0,
\end{equation}
and proceed by contradiction. This implies that there exists
a subsequence $w_k \rightarrow r$ such that
\begin{equation}
\lim \limits_{w_k \rightarrow r} \|\hat{\Psi}(w_k)\| = 0.
\end{equation}
Hence, the individual components $\hat{\Psi}_j$ must also be zero in the same limit.
From the continuity of the norm
and $\phi$, we then have that
\begin{align}
\lim \limits_{w_k \rightarrow r} |\phi(H_j(w_k; r), G_j(F, w_k; r))|
=
|\phi(\alpha,\beta)|=0,
\end{align}
where
\begin{align}
\alpha=\lim\limits_{w_k \rightarrow r}H_j(w_k; r) \text{ and } \beta=\lim \limits_{w_k \rightarrow r} G_j(F, w_k; r).
\end{align}
This can happen if and only if $0\leq\alpha\perp\beta\geq 0$. Therefore, at least one of $\alpha$ or $\beta$ must be $0$. Hence, for every sequence converging to $r$, there exists a subsequence $w_k$ such that at least one of the following holds,
\begin{align}
\lim \limits_{w_k \rightarrow r} |H_j(w_k; r)| = 0 \text{ or } \lim \limits_{w_k \rightarrow r}|G_j(F, w_k; r)|=0,
\end{align}
which contradicts \eqref{eqn:star1} and \eqref{eqn:star2}.
\end{proof}

It remains to construct a concrete instance of a complementarity deflation operator.
The operator used in our study is related to standard norm deflation. To
construct it, we first need to introduce the compactly supported
$C^\infty$ test function,
\begin{equation}
\label{eq:CHI}
\chi(z)=\begin{dcases}\begin{array}{ll}
\exp{\left(1+\dfrac{\delta}{\|z\|-\delta}\right)}, & \text{if }\|z\|<\delta,\\
0 & \text{if }\|z\|\geq\delta,
\end{array}
\end{dcases}
\end{equation}
Note that $\chi(0)=1$. The value $\delta$ is chosen in $(0, \delta_r]$ as in definition
\ref{def:compl_defl_op}. In practice, $\delta_r$ is not known in advance,
so $\delta$ is set to a small value. In the examples presented later we use
$\delta = 10^{-6}$ unless otherwise mentioned. The complementarity deflation operator
we consider is
\begin{equation}
\label{eq:NCP_defl_op}
H_j(z; r)=\frac{z_j+\chi(z-r)}{\mathlarger{\|z-r\|^{p}}},\hspace{12pt}\text{for all }z\in\mathscr{F}\setminus\{r\},
\end{equation}
where $p\geq 1$ is the power of the complementarity deflation operator.

\begin{lemma}
Operator \eqref{eq:NCP_defl_op} is a complementarity deflation operator.
\end{lemma}

\begin{proof}
Away from the deflated solution $r$, the function $\chi$ vanishes and we are
left with the position vector divided by the norm to some power. As the norm is positive, the
sign of the position vector is always preserved away from $r$. Let us now
consider the points close to $r$, i.e. in the open set $\mathscr{F} \cap
B(r,\delta)$. Here $\chi$ is always positive, and $z_j \ge 0$ as $z \in
\mathscr{F}$. Hence, for all points
in the open set $\mathscr{F} \cap B(r,\delta)$, we have that for all $j$,
\begin{equation}
\mathrm{sign}\left[H_j(z)\right] > 0.
\end{equation}
Hence, property \eqref{eq:comp_defl_sign_preservation} holds. Furthermore,
\begin{align}
\liminf\limits_{z_k\rightarrow r}\left|H_j(z)\right| &= \liminf\limits_{z_k\rightarrow r} H_j(z) \\
                                                   &= \underbrace{\liminf\limits_{z_k\rightarrow r} \frac{z_j}{\mathlarger{\|z-r\|^{p}}}}_{\ge 0} + \underbrace{\liminf\limits_{z_k\rightarrow r} \frac{\chi(z-r)}{\mathlarger{\|z-r\|^{p}}}}_{> 0} \\
                                                   &> 0.
\end{align}
Hence, property \eqref{eq:compl_defl_bounded_away} holds also.
\end{proof}

\subsection{Deflating several roots}
This strategy can be repeatedly applied to eliminate several solutions. Assume $m$ distinct solutions $r^1, ..., r^m$ are available.
Consider the composition of deflation operators
\begin{align}
\hat{G}(F,z) &= G(F,z;r^m) \circ_1 \cdots \circ_1 G(F,z;r^1),\notag\\
\hat{H}(z) &= H(z;r^m) \circ_1 \cdots \circ_1 H(z;r^1),
\end{align}
where $G(F,z;r^i)$ is a weak complementarity deflation operator, $H$ is a complementarity deflation operator
and $\circ_1$ indicates the composition with respect to the first argument, i.e.
\begin{align*}
G(F,z;s) \circ_1 G(F,z;r)= G(G(F,z;r),z;s).
\end{align*}
The deflated NCP residual is then
\begin{align}
\hat{\Psi}(z) = \Phi(\hat{H}(z), \hat{G}(F,z)).
\end{align}
An induction argument shows that multiple deflation is still effective if we consider sequences
$z_k\in\mathscr{F}\setminus\{r^1,...,r^m\}$ converging to any of the deflated roots: the norm
of $\hat{\Psi}$ will be bounded away from zero in the limit inferior. Similarly,
other solutions are preserved under the assumption that all roots are isolated.

When deflating multiple solutions it is important to apply a shift to the deflation operator.
For clarity of the argument, take $p=1$.
Away from previously found solutions, $\chi$ vanishes and we have
\begin{align*}
\hat{G}(F,z) = \frac{F(z)}{\prod\limits_i\|z - r^i\|} \text{ and } \hat{H}(z) = \frac{z}{\prod\limits_i\|z - r^i\|}.
\end{align*}
As any of the terms $\|z - r^i\|$ go to infinity, the deflated position vector and residual go to zero,
thus causing solvers to erroneously report convergence at points far from the deflated solutions.
As the denominators are the product of such terms, multiple deflation aggravates this behaviour.
This problem is solved by adding a shift to the deflated arguments. By defining the deflated
residual as
\begin{align}
\hat{\Psi}(z) = \Phi(\hat{H}(z) + \alpha z, \hat{G}(F,z) + \alpha F),
\end{align}
with $\alpha> 0$, as the product of the $\|z - r^i\|$ goes to infinity,
the deflated position vector and residual asymptotically behave as $\alpha z$ and $\alpha F$ respectively,
and hence the solver does not erroneously report convergence to spurious solutions.
The choice $\alpha = 1$ is natural, but it is sometimes advantageous to choose different values.
It is straightforward to see that shifting preserves the properties of the deflation operators. 

\section{Extension to MCPs}
It is possible to extend the semismooth rootfinding formulation to
MCPs, and to apply deflation in this case also.
Given MCP$(F, l, u)$ and $J = \{1, \dots, n\}$, define the index sets
\begin{align}
\label{eq:index_sets}
\begin{array}{ll}
J_{l\phantom{u}} := \{i\in J : -\infty<l_i<u_i=+\infty\},& J_u:=\{i\in J: -\infty=l_i<u_i<+\infty\},  \\
J_{lu}:= \{i\in J: -\infty<l_i\leq u_i<+\infty\},& J_f:=\{i\in J: -\infty=l_i<u_i=+\infty\}.  
\end{array}
\end{align}
$J_l$ and $J_u$ represent the variables with only lower and upper bounds respectively, 
$J_{lu}$ represents the variables with both lower and upper bounds, and $J_f$ represents
the free variables with no bounds.
The solutions of the MCP are the roots of the associated
MCP residual \cite{billups1995}
\begin{equation}
\Psi_i(z) = \Phi_i(z, F(z)) := 
\begin{dcases}
\phi(z_i-l_i,F_i(z)), & \text{ if } i\in J_l,\\
-\phi(u_i-z_i,-F_i(z)), & \text{ if } i\in J_u,\\
\phi(z_i-l_i,\phi(u_i-z_i,-F_i(z))), & \text{ if } i\in J_{lu},\\
-F_i(z), & \text{ if } i\in J_f.
\end{dcases}
\end{equation}
In this case, the deflated operator can be constructed via
\begin{equation}
\hat{\Psi}_i(z) :=
\begin{dcases}
\phi(H_i(z - l; r), G_i(F, z; r)), & \text{ if } i\in J_l,\\
-\phi(H_i(u-z; r), -G_i(F, z; r)), & \text{ if } i\in J_u,\\
\phi(H_i(z - l; r), \phi(H_i(u - z; r),-G_i(F, z; r))), & \text{ if } i\in J_{lu},\\
-G_i(F, z; r), & \text{ if } i\in J_f.
\end{dcases}
\end{equation}

By the properties of $G$, there exists at least one $j$ for which
\begin{equation}
\liminf \limits_{z_k \rightarrow r} \left| G_j(F, z; r) \right| > 0.
\end{equation}
The argument proceeds by considering the cases $j \in J_l, j \in J_u, j \in J_{lu}, j \in J_f$ in turn.
No matter which case applies, the strong complementarity operator
ensures that the deflated residual will not converge to zero as $z_k \rightarrow r$, as $u - z$ and
$z - l$ are both nonnegative vectors, as occur with NCPs.

\section{Examples}
\begin{table}
\centering
\begin{tabular}{c|c|c|c|c|c}
\toprule
{Problem} & dim & \# sol & \# sol found & \# iters & ${(p,\alpha)}$\\
\midrule[1pt]
Kojima and Shindoh (1986) & 4   & 2  & 2 & (9, 11) & (1.0, 0.5)\\
\hline
\rule{0pt}{2.5ex}
Aggarwal (1973) & 3 & 3 & 3 & (8, 5, 10) & (1.0, 1.0)\\
\hline
\rule{0pt}{2.5ex}
Konno and Kuno (1992) & 9 & 3 & 3 & (6, 22, 11) & (1.0, 0.5) \\
\hline
\rule{0pt}{2.5ex}
Gould (2015) & 4 & 3 & 3 & (5, 4, 4) & (2.0, 1.0) \\
\hline
\rule{0pt}{2.5ex}
Tin-Loi (2003) & 42 & 2 & 2 & (9, 11) & (1.0, 1.0) \\
\hline
\rule{0pt}{2.5ex}
Mathiesen (1987) & 4 & $\infty$ & 100 & 3--7 & (1.0, 1.0) \\
\bottomrule
\end{tabular}
\caption{Summary of the test problems. dim is the problem dimension, \# sol is the number of
known solutions, \# sol found is the number of solutions found, \# iters is the number of nonlinear iterations needed to find
each solution, and $(p, \alpha)$ are the deflation parameters used. All
solutions were found to an $\ell_2$ residual tolerance of less than $10^{-10}$.
Apart from the Mathiesen (1987) problem, our algorithm found all solutions of
each problem using only one initial guess in under two seconds on a laptop.}
\label{tab:algebraic_problems}
\end{table}

In this section we consider six NCP test problems. These tests are important as
they show whether deflation is actually useful in practice for finding
additional solutions. The test problems we have chosen are small algebraic
problems that span the range of applications of complementarity problems and are
significant in the literature for different reasons. The characteristics of
the problems and of our computational results are shown in table
\ref{tab:algebraic_problems}.

The source code for all examples is included in the supplementary
material.

\subsection{Kojima and Shindoh (1986)}
This problem was first proposed by Kojima and Shindoh \cite{kojima1986} and
is an NCP with $F:\mathbb{R}^4\rightarrow\mathbb{R}^4$ given by
\begin{align}
F(z)=\left[\begin{array}{l}
3z_1^2 + 2z_1z_2 + 2z_2^2 + z_3 + 3z_4 - 6\\
2z_1^2 + z_2^2   + z_1 + 10z_3 + 2z_4 - 2\\
3z_1^2 + z_1z_2  + 2z_2^2 +  2z_3 + 9z_4 - 9\\
 z_1^2 + 3z_2^2  + 2z_3 + 3z_4 - 3
\end{array}\right].
\end{align}
It admits two solutions,
\begin{align*}
\begin{array}{lclcl}
\bar{z}^1 = [1, 0, 3, 0]^T, & & \multirow{2}{*}{\text{with residuals}} & & F(\bar{z}^1)=[0, 31, 0, 4]^T,\\
\bar{z}^2 = [\sqrt{6}/2, 0, 0, 1/2]^T,& & & & F(\bar{z}^2)=[0,2+\sqrt{6}/2, 0, 0]^T.
\end{array}
\end{align*}
This problem was used again by Dirkse and Ferris \cite{dirske1995} as an
example of a problem in which classical Newton solvers struggle to find a solution.
This is because one of the two solutions, $\bar{z}^2$, has a degenerate third
component, i.e. $\bar{z}^2_3=F(\bar{z}^2)_3=0$. This causes
$\Phi_3(\bar{z}^2)=\phi_{FB}(\bar{z}^2_3,F_3(\bar{z}^2))=\phi_{FB}(0,0)$. The
nondifferentiability of $\phi_{FB}$ induces nonexistence of the
Jacobian of $\Phi$ at the origin. Another feature of this problem is that the linear
complementarity problem formed through linearisation of the residual $F$ around
zero has no solution, causing difficulties for the Josephy--Newton method there
\cite{josephy1979}.

This is a relatively easy problem to solve and deflation successfully finds both
solutions with many combinations of power, shift and initial guess. We chose $p=1$,
$\alpha=0.5$ and initial guess $[2, \dots, 2]^T$.

\subsection{Aggarwal (1973)}
This is a Nash bimatrix equilbrium problem arising in game theory.
This kind of problem was first introduced by von Neumann and
Morgenstern \cite{vonneumann1945} and the existence of its solutions was
further studied by Nash \cite{nash1951} and Lemke and Howson
\cite{lemke1964}. In the same paper, Lemke and Howson also presented a numerical
algorithm for computing solutions to these kinds of problems. This particular
example was introduced by Aggarwal \cite{aggarwal1973} to prove that it is
impossible to find all solutions of such problems using a modification of
the Lemke--Howson method that had been conjectured to compute all solutions.

The problem consists of finding the equilibrium points of a bimatrix (non-zero
sum, two person) game. Let $A$ and $B$ be the $n \times n$ payoff matrices of
players $1$ and $2$ respectively. Let us assume that player $1$ plays the
$i^{th}$ pure strategy and player $2$ selects the $j^{th}$ pure strategy amongst
the $n$ strategies available to each. The entries of $A$ and $B$, $a_{i,j}$ and
$b_{i,j}$ respectively, correspond to the payoff received by each player. It is then possible
to define a mixed strategy for a player which consists of a $n \times 1$ vector
$x$ such that $x_i\geq0$ and $x_1+...+x_n=1$. Denote by $x$ and $y$ the mixed
strategies for player $1$ and $2$ respectively. The entries of these vectors
stand for the probability of the player adopting the corresponding pure
strategy.  The expected payoffs of the two players are then $x^TAb$ and $x^TBy$
respectively. An equilibrium point $(x^*,y^*)$ is reached when, for all $x$,
$y$,
\begin{align}
(x^*)^TAy^*\geq x^TAy^*,\hspace{12pt}\text{and}\hspace{12pt}(x^*)^TBy^*\geq (x^*)^TBy,
\end{align}
i.e. neither player can unilaterally improve their payoff.

We will consider Aggarwal's counterexample
\cite{aggarwal1973}, which admits three Nash equilibria.
These equilibria are related to the solutions of the NCP with residual
\begin{align}
F(z) =
\begin{pmatrix}
\bar{A}y - e \\
\bar{B}^Tx - e
\end{pmatrix}
,
\end{align}
where $z = [x, y]^T$ and $e = [1, 1, \dots, 1]^T$,
$\bar{A}$ and $\bar{B}$ are positive-valued loss matrices related to $A$ and $B$
respectively, and $x$ and $y$ relate to the mixed strategy adopted by each player \cite[\S 1.4]{murty1988}.
The data for this problem is
\begin{align*}
\bar{A} =
\begin{bmatrix}
30 & 20 \\
10 & 25
\end{bmatrix},
\text{ and }
\bar{B} =
\begin{bmatrix}
30 & 10 \\
20 & 25
\end{bmatrix}.
\end{align*}
The initial guess we considered was $[0,0,0,1/30]$, with
$(p,\alpha)=(1, 1)$, although other choices yield the same results. The
three solutions found are
\begin{align*}
\begin{array}{lclcl}
\bar{z}^1 = [0, 1/20, 1/10, 0]^T, & & \multirow{4}{*}{\text{with residuals}} & & F(\bar{z}^1)=[2, 0, 0, 1/4]^T,\\
\bar{z}^2 = [1/110, 4/110, 1/110, 4/110]^T,& & & & F(\bar{z}^2)=[0, 0, 0, 0]^T,\\
\bar{z}^3 = [1/10, 0, 0, 1/20]^T, & & & & F(\bar{z}^3)=[0, 1/4, 2, 0]^T.\\
\end{array}
\end{align*}
Aggarwal observed that the conjectured scheme mentioned above could compute
$\bar{z}^1$ and $\bar{z}^3$, but could not compute $\bar{z}^2$.

\subsection{Konno and Kuno (1992)}
This problem, proposed by Konno and Kuno \cite{konno1992}, is a linear
multiplicative programming problem, i.e. a problem in the form
\begin{align}
\min\limits_x f(x)=(c^Tx+c_0)\cdot(d^Tx+d_0)\hspace{12pt}\text{s.t.}\hspace{12pt}Ax\geq b,
\end{align}
where $A \in \mathbb{R}^{m\times n}$, $b \in \mathbb{R}^m$, $c$, $d$, $x \in
\mathbb{R}^n$ and $c_0$, $d_0 \in \mathbb{R}$. As this is an optimization
problem with inequality constraints, its optimality system can be reformulated as an MCP. The
first order KKT optimality conditions can be written as
\begin{align}
\nabla f - A^T\lambda = 0,\hspace{12pt} \lambda_i\geq 0,\hspace{12pt}(Ax-b)_i\geq 0,\hspace{12pt}\lambda_i(Ax-b)_i=0\hspace{12pt}\text{for all }i,
\end{align}
where $\lambda\in\mathbb{R}^m$. This is equivalent to MCP$(F(z), l, u)$ with $z\in\mathbb{R}^{n+m}$ and
\begin{align}
z = \left[\begin{array}{c}x \\\hline\lambda\end{array}\right],\hspace{10pt}l=\left[\begin{array}{c}-\infty \\\hline 0\end{array}\right], \hspace{10pt}u=\left[\begin{array}{c}+\infty \\\hline +\infty \end{array}\right],\hspace{10pt}F=\left[\begin{array}{c}\nabla f - A^T\lambda \\\hline Ax-b \end{array}\right].
\end{align}
The problem proposed by Konno and Kuno makes the choice $n=2$, $m=7$, $c^T=[1,1]$, $d^T=[1,-1]$, $c_0=d_0=0$, and
\begin{align}
A = 
\begin{bmatrix}
-1/5 & -2/5\\1/4 & -7/25\\7/20 & 7/20\\14/25 & 7/25\\7/12 & 0\\-28/65 & 7/65\\-14/31 & -7/31
\end{bmatrix},
\text{ and }
b = 
\begin{bmatrix}
6/5\\ 21/25\\ 7/10\\ 14/25\\ 7/12\\ 84/65\\ 42/31
\end{bmatrix}.
\end{align}

Our solver implementation currently supports only NCPs.
For this reason we applied the change of variables 
\begin{align}
z'_i=z_i+5\hspace{12pt}\text{if}\hspace{12pt}1\leq i\leq 2;\hspace{24pt}z'_i=z_i\hspace{12pt}\text{if}\hspace{12pt}3\leq i\leq 9,
\end{align}
to move the feasible region to the first quadrant of $\mathbb{R}^2$. This does not change the
complementarity problem as the bounds are enforced anyway from the inequality
constraints. The problem, thus formulated as NCP$(F(z'))$, can now be solved
using our method. These problems are in general quite difficult; linear
multiplicative programming problems are NP-hard \cite{matsui1996,benson1997}.
However, we were able to find all the solutions starting from the same
initial guess $[1/10, 36/10, 0, \dots, 0]^T$ and with $(p,\alpha)=(1,0.5)$
(other choices yield the same results). The three solutions found for
this problem (in the original coordinate system) are
\begin{flalign}
& \bar{z}^1 = [0, 0, 0, 0, 0, 0, 0, 0, 0]^T,\\
& \bar{z}^2 = [-2, 4, 0, 0, 144/7, 0, 0, 52/7, 0]^T,\\
& \bar{z}^3 = [0, -3, 10, 50/7, 0, 0, 0, 0, 0]^T.
\end{flalign}
with residuals
\begin{flalign}
& F(\bar{z}^1)=[0, 0, 6/5, 21/25,7/10,14/25,7/12,84/65,42/31]^T,\\
& F(\bar{z}^2)=[0, 0, 12/5, 63/25,0,14/25,7/4,0,42/31]^T,\\
& F(\bar{z}^3)=[0, 0, 0, 0,7/4,7/5,7/12,21/13,21/31]^T.
\end{flalign}

\subsection{Gould (2015)}
This is a nonconvex quadratic programming problem with linear constraints suggested by N.\
I.\ M.\ Gould in personal communication. It is a quadratic minimisation problem
with an indefinite Hessian of the form
\begin{align*}
\min\limits_x f(x)=-2(x_1-1/4)^2+2(x_2-1/2)^2,\hspace{15pt}\text{s.t.}\hspace{15pt}\left\{\begin{array}{r}x_1+x_2\leq 1,\\6x_1+2x_2\leq 3,\\x_1,x_2\geq 0.\end{array}\right.
\end{align*}
Such problems are known to be NP-complete \cite{vavasis1991}. The first order KKT
optimality conditions yield an NCP with residual
\begin{align}
F(z) = \left[\begin{array}{r}-4(x_1-1/4) + 3\lambda_1+\lambda_2\\4(x_2-1/2)+\lambda_1+\lambda_2\\3-6x_1-2x_2\\1-x_1-x_2\end{array}\right],
\end{align}
where $z=[x,\lambda]$, with $\lambda=[\lambda_1,\lambda_2]$ the vector of the
Lagrange multipliers associated with $F_3(z)\geq0$ and $F_4(z)\geq0$
respectively. Note that in this case it is not necessary to use Lagrange
multipliers to enforce $x\geq0$ as this is implicit in the NCP formulation. The
nonconvexity of the function $f$ makes this problem difficult; it attains two
minima with similar functional values and has a saddle point at $x=[1/4,1/2]^T$. The
standard method for solving a problem of this kind is an interior point method
that searches for a central path from the initial guess to the solution via
continuation on a barrier function parameter \cite{nocedal2006}; in this
case, the central path is pathological, with different paths converging to the
different minima.

We directly solve the arising NCP with the semismooth Newton method and deflation.
The initial guess used was
$[3/10, \dots, 3/10]^T$ and the deflation parameters used were
$(p,\alpha)=(2, 1)$ (other choices yield the same results).
The three solutions found were
\begin{align*}
\begin{array}{lcl}
\bar{z}^1 \mathrlap{\text{ (global) }}\phantom{\text{ (saddle) }}= [0, 1/2, 0,0]^T, & \multirow{3}{*}{\text{with residuals}} & F(\bar{z}^1)=[1, 0, 1, 1/2]^T,\\
\bar{z}^2 \text{ (saddle) }= [1/4, 1/2, 0, 0]^T,& & F(\bar{z}^2)=[0, 0, 1/4, 1/4]^T,\\
\bar{z}^3 \mathrlap{\text{ (local) }}\phantom{\text{ (saddle) }}= [11/32, 15/32, 1/8, 0]^T, & & F(\bar{z}^3)=[0, 0, 0, 3/16]^T.
\end{array}
\end{align*}
The KKT conditions make no distinction between minima and saddle points, and
hence the solver finds both kinds of stationary points.

\subsection{Tin-Loi and Tseng (2003)}
This is a quasibrittle fracture problem from the MCPLIB collection of
MCPs collated by Dirske and Ferris \cite{dirske1995}. Such problems frequently
support multiple solutions, and computing them is of physical importance
\cite{bolzon1997,tinloi2003}.
This linear complementarity problem was suggested to us in personal communication
by T.\ S.\ Munson. The residual is of the form
\begin{align}
F(z) = Az + b,
\end{align}
with $A \in \mathbb{R}^{n \times n}$, $b \in \mathbb{R}^n$, and $n = 42$. The
data file from MCPLIB that defines $A$ and $b$ is included in the supplementary material.

We chose $[2/5, \dots, 2/5]^T$ as the initial guess and used
$(p, \alpha) = (1, 1)$. For this problem we used the
standard projected Armijo linesearch and not the linesearch of Sun et al.\ 
\cite{sun2002}, as this proved to be more efficient for this problem. Both known
solutions were found.

\subsection{Mathiesen (1987)}
\label{sec:Mathiesen}
This example considers the computation of a Walrasian equilibrium
\cite{arrow1954}. Under strong assumptions on consumer preferences, such
problems have unique solutions; in general, these problems permit multiple
equilibria, a fact with significant consequences for general equilibrium theory
\cite{rizvi2006}.  This problem was first proposed by Mathiesen
\cite{mathiesen1987}, and was used as a test problem by Mangarasian and Solodov
\cite{mangasarian1993} and by Pang and Gabriel \cite{pang1993}. This problem can
be formulated as an NCP with residual
\begin{align}
F(z)=\left[\begin{array}{r} -z_2 + z_3 +z_4\\ z_1 - 0.75(z_3+\gamma z_4)/z_2\\-z_1-0.25(z_3+\gamma z_4)/z_3+1\\\gamma-z_1\end{array}\right],
\end{align}
where $\gamma$ is a positive parameter. The reason we chose this problem is that
it admits a continuum of solutions. We can thus test the behaviour of
our deflation technique on a case in which all the equilibrium points cannot be
found in finite time, and where the isolated solutions assumption does not hold. The parameter
$\gamma$ determines the type of the solutions of this problem. In particular,
if $\gamma < 3/4$, the solutions are characterized by
\begin{align*}
\bar{z} =
\begin{bmatrix}
\gamma \\ 3\lambda(1-\gamma)/\gamma \\ \lambda \\ 2\lambda(1-\gamma)/\gamma-\lambda
\end{bmatrix}
\text{ with residual }
\mathrlap{
F(\bar{z}) =
\begin{bmatrix}
0 \\ 0 \\ 0 \\ 0
\end{bmatrix},
}
\phantom{
F(\bar{z}) =
\begin{bmatrix}
0 \\ 0 \\ 0 \\ \gamma - 3/4
\end{bmatrix},
}
\end{align*}
and if $\gamma > 3/4$,
\begin{align*}
\mathrlap{
\bar{z} =
\begin{bmatrix}
3/4 \\ \lambda/2 \\ \lambda/2 \\ 0
\end{bmatrix}}
\phantom{
\bar{z} =
\begin{bmatrix}
\gamma \\ 3\lambda(1-\gamma)/\gamma \\ \lambda \\
2\lambda(1-\gamma)/\gamma-\lambda
\end{bmatrix}
}
\text{ with residual }
F(\bar{z}) =
\begin{bmatrix}
0 \\ 0 \\ 0 \\ \gamma - 3/4
\end{bmatrix},
\end{align*}
for any $\lambda > 0$ \cite{pang1993}.

Additionally, there exist sequences $z_k \rightarrow 0$ such that
\begin{equation*}
\lim\limits_{z_k \rightarrow 0}F(z_k) = [0, 0, 1, 0]^T,
\end{equation*}
which satisfies the conditions to be satisfied at a solution. For example, take $z_k = [s_k, s_k, 0, 0]^T$,
$s_k \downarrow 0$. Hence, although $F$ is
singular at zero, the solver might erroneously converge there. Unfortunately, this
happens in practice, and the solver eventually terminates with a division by zero error.
To obtain our results, we first deflated this point before actually solving the problem. 
This is similar in spirit to the approach advocated by Kanzow \cite{kanzow2000};
it successfully prevents the solver iterates from
approaching it. With the choice $\gamma=1$,
$(p,\alpha, \delta)=(1,1, 10^{-8})$ and initial guess $[15, \dots, 15]^T$, the solver found $100$
different solutions (in around $5.5$ seconds), corresponding to $100$ different
values of $\lambda$.

We would like to emphasize again that this problem does not have isolated
solutions, a requirement of our earlier theory. Having isolated solutions is not a necessary
condition for the elimination of known solutions; instead, it is a necessary condition
for ensuring that all existing solutions are preserved. Every time deflation was
applied in this problem, all the other solutions in the support of
$\chi$ from the deflated problem were removed. This is due to the fact that we artificially
add a positive value to all the entries of the position vector in the
neighbourhood of the solution; after this, they too are eliminated from the
deflated problem. Nevertheless, this example demonstrates the strength of the
approach: the combination of the semismooth Newton solver and deflation
identified a hundred solutions before failing to converge.

\section{Conclusion}
In this paper we constructed a new class of deflation operators, complementarity
deflation operators, and a corresponding theory that guarantees their success.
We showed that by applying these operators to the arguments of a semismooth
reformulation of an NCP, its residual is prevented from converging to zero for
any sequence converging to the known solution.

The effectiveness of the approach was demonstrated through its application to
several problems that support multiple solutions. For all problems with a finite
solution set, we identified an initial guess and parameter values that finds all
known solutions. However, it would be desirable to design robust complementarity
deflation operators without parameters, or to devise a rigorously grounded
scheme for choosing them.

An important future extension of these ideas is to the infinite-dimensional
case \cite{ulbrich2011}. These arise in many important applications, including in optimization
constrained by partial differential equations and inequality constraints.

\bibliographystyle{siam}
\bibliography{literature}

\end{document}